\documentclass[a4paper,11pt]{article} 

\usepackage{t1enc} 
\usepackage[utf8]{inputenc}
\usepackage{amssymb,amsmath,amsthm} 
\usepackage{graphicx,hyperref}
\usepackage{pgf,tikz}
\usepackage{hyperref}
\usepackage{color}
\usetikzlibrary{shapes,backgrounds,arrows}
\usepackage{mathrsfs}
\usetikzlibrary{arrows}
\newtheorem{theorem}{Theorem}[section]
\newtheorem{proposition}[theorem]{Proposition}
\newtheorem{corollary}[theorem]{Corollary}
\newtheorem{lemma}[theorem]{Lemma}
\newtheorem{definition}[theorem]{Definition}

\title{On balanced biregular cages}
\author{}

\begin{document}

\author{Gabriela Araujo-Pardo\textsuperscript{1}
 \\
garaujo@im.unam.mx
\\
Gy\"orgy Kiss\textsuperscript{1,2}
 \\
gyorgy.kiss@ttk.elte.hu
}

\maketitle
\begin{abstract}
In this paper, we introduce a problem closely related to the {\emph{Cage Problem}}. We are interested in {\emph{Balanced Biregular Cages}}, which are the smallest biregular graphs of fixed girth that have the same number of vertices of one degree as the other. We provide lower and upper bounds on the order of balanced biregular graphs. In particular, we construct relatively small balanced biregular graphs from incidence graphs of finite projective, affine, and biaffine planes and we show that some of the obtained graphs are balanced biregular cages.
\end{abstract}

\footnotetext[1]{Research supported in part by PAPIIT-UNAM-M{\' e}xico IN113324, and  SECIHTI-M{\'e}xico under Project CBF2023-2024-552.}
\footnotetext[2]{This research was supported in part by the Hungarian National Research, Development and Innovation Fund, grant numbers 2024-1.2.5-T\'ET-2024-00082 and SNN 152582, and by the Slovenian Research Agency, research program N1-0429.}


\section{Introduction}
Many issues in extremal graph theory involve finding the maximum or minimum of a graph parameter subject to constraints on other parameters or structural properties. Often, these studies focus on estimating key quantities such as the number of vertices or edges. In this paper, we explore similar questions, which are also inspired by the classic {{\em{Cage Problem}}, which consists of finding regular graphs of fixed girth (length of the smallest cycle) and minimum order. 

There exist a lot of different generalizations of this problem, one of them asks about graphs with a given degree set, a fixed girth, and minimum order. These graphs were introduced by Chartrand, Gould, and Kapoor in \cite{CGK81}, and also in that paper, they focus on the question of biregular graphs (where the degree set has exactly two elements). Also, Yuansheng and Liang \cite{YuLi03} studied these biregular graphs and gave interesting lower and upper bounds for some specific girths, one of which is the lower bound for biregular graphs of girth $6$. After that, many authors have studied bipartite cages and obtained different results regarding their order, see \cite{AraExoJaj, BobJajPis, ExooJaj16, Fil17, JWGJS}.

Later, Filipovski, Jajcay, and Ramos-Rivera \cite{FilRamRivJaj19} further specialized the problem for bipartite biregular graphs, that is, for bipartite graphs with the property that the vertices of one bipartition set have degree $m$ and the vertices of the other bipartition set have degree $n$. Note that, since the graphs are bipartite, the girth is even. We would like to emphasize that one of the differences between the problems of finding biregular and bipartite biregular cages is related to the "balance" between the vertices of one degree and the other. While in the problem of finding biregular cages, it is essential to have only a small number of vertices of higher degree, in the case of bipartite biregular graphs, it is crucial to consider arithmetic relationships between the two degree values that allow one side of the partition to consist of vertices of one degree and the other side of vertices of the other degree.

In this paper, we introduce a third variation of the biregular cage problem, which we refer to as the balanced problem. We ask about the existence of {\emph{balanced biregular cages}}, defined as biregular graphs with a fixed girth and minimum order, in which the number of vertices of one degree is equal to the number of vertices of the other. 

The paper is organized as follows. In Section \ref{preliminaries}, we recall the classical {\emph{Cage Problem}} and the well-known {\emph{Moore bound}} associated with it. We also introduce the notion of {\emph{balanced biregular cages}}, along with the definition of finite projective planes, their incidence graphs, and the concept of biaffine planes. 

In Section \ref{bounds}, we establish the existence of these cages and provide lower and upper bounds 
on their order. We also address the case in which the cage contains leaves (vertices of degree $1$). In Section \ref{girth3and4}, we completely solve the problem for graphs of girth $3$ and $4$. Section \ref{girth5} is devoted to girth $5$, where we present specific constructions of balanced biregular cages with consecutive degrees, as well as a general construction of bipartite biregular graphs using the Levi graphs of projective planes and biaffine planes.

In Section \ref{girth6}, we construct balanced biregular cages and graphs of girth $6$ using projective planes. Finally, in Section \ref{conclandfuturework}, we conclude the paper with a discussion of the results and directions for future work. 

\section{Preliminaries} \label{preliminaries}
First of all, we recall the traditional  {\emph{Cage Problem:}} A simple, finite, connected graph $G=(V, E)$ is \emph{$k$-regular} if each vertex has exactly 
$k$ neighbours. It is of  \emph{girth $g$} if its smallest cycles have $g$ vertices. A 
\emph{$g$-cycle} or  \emph{girth cycle} is a cycle of length $g$. The number 
of vertices of $G$ is called the \emph{order} of $G$. A \emph{$(k,g)$-graph} is a
$k$-regular graph of girth $g$. If it has minimal order, then it is called a 
\emph{$(k,g)$-cage}. For more information about cages, we refer to the dynamic 
survey by Exoo and Jajcay \cite{ExooJaj08}.

For $k=2$ the connected $(k,g)$-graphs are the $g$-cycles. For $k>2$, the well-known \emph{Moore bound} is obtained by counting the vertices of a tree that emerge from a vertex if $g$ is odd, or from an edge if $g$ is even. As we will use these trees, called  {\em{Moore trees}}, and their
modifications on several occasions throughout this work, we provide a brief description of them.

Let $u$ be a vertex and let ${\cal{T}}_{u}$ be a tree of depth $\frac{g-1}{2}$, rooted in $u$, with all the vertices of degree $k$ except the leaves. Notice that, if $G$ is a $k$-regular graph of odd girth $g$, then $G$ necessarily contains this tree, denoted by  ${\cal{T}}_u$.
On the other hand, if $g$ is even and $uw$ is an edge, then a Moore tree for even girth, denoted by 
${\cal{T}}_{uw}$, consist of the edge $uw$ and two disjoint trees, ${\cal{T}}_u$ and ${\cal{T}}_w$ rooted in $u$ and $w$, respectively. Both ${\cal{T}}_u$ and ${\cal{T}}_w$ have depth $\frac{g-2}{2}$ and 
all the vertices have degree $k$, except the leaves. As before, if $G$ is a $k$-regular graph of even girth $g$, then $G$ must contain ${\cal{T}}_{uw}$. 
Counting the vertices of these Moore trees results in the following bound.
    \begin{theorem}[Moore bound, \cite{ExooJaj08}]
    Let $G$ be a $(k,g)$-graph with $k>2$. Then the order of $G$ is at least $n_0(k,g)$, where
    $$n_0(k,g)= 
\begin{cases}
\frac{k(k-1)^{\frac{g-1}{2}}-2}{k-2}, \quad\text{if g is odd;}\\
\frac{2(k-1)^\frac{g}{2}-2}{k-2}, \quad\text{~~if g is even.}
\end{cases}$$
    \end{theorem}}

In this paper, we will study the \emph{balanced biregular graphs and cages}, defined as follows:
\begin{definition}
Let $1\leq r\leq s$ and $3\leq g$ be integers. An {\emph{$(r,s;g)$-balanced biregular graph}}
(for short an $(r,s;g)$-babi-graph) is a graph of girth $g$ having degree set equal to $\{ r,s\} $ and satisfying the additional property that the number of vertices of degree $r$ equals the number of vertices of degree $s$. We denote it as 
$(r,s;g)$-babi-graph and call it, for short, a babi-graph. An $(r,s;g)$-{\emph{balanced biregular cage}} (for short an $(r,s;g)$-babi-cage) is an $(r,s;g)$-babi-graph of minimum order. Moreover, $n_{bb}(r,s;g)$ denotes the order of an 
$(r,s;g)$-babi-cage.
\end{definition}


A significant part of our constructions uses the geometric properties of finite planes. Our motivation comes from the fact that the incidence graphs of projective planes are known to be extremal with regard to the $(k,6)$-graphs. An introduction to finite geometries and the detailed description of concepts we use can be found, for example, in \cite{KSz}. Here we give only the most necessary definitions.

A finite projective plane of order $q$ has $q^2+q+1$ points and $q^2+q+1$ lines, and it satisfies the following four axioms:
\begin{itemize}
\item{\bf P1.} 
For any two distinct points, there is precisely one line 
that is incident with both points.
\item{\bf P2.} 
For any two distinct lines, there is precisely one point that is incident with both lines.
\item{\bf P3.}
Each line is incident with exactly $q+1$ points.
\item{\bf P4.}
Each point is incident on exactly $q+1$ lines..
\end{itemize}
For each prime power $q$, there exists $\mathrm{PG}(2,q)$, a projective plane of order $q$, that can be coordinatized by the elements of $\mathrm{GF}(q)$, the finite field of order $q$. All the known projective planes have prime power order. It is an open problem whether there is a projective plane whose order is not a prime power.

One can also construct other incidence structures from projective planes. We will use biaffine planes in this paper. These are not so well-known as projective planes, so for the sake of completeness, we give their definition.

\begin{definition}
Let $\Pi$ be a finite projective plane of order $q$. A biaffine plane is obtained from $\Pi$ by choosing a point-line pair $(P,\ell )$ and deleting $P$, $\ell$, all the lines incident with $P$, and all the points belonging to $\ell$. If the point-line pair is incident in $\Pi$, then we call the biaffine plane type 1; otherwise, type 2.
\end{definition}

The \emph{incidence graph}, also called the \emph{Levi graph}, of a point-line incidence structure is a bipartite graph whose bipartition sets correspond to the
set of points and set of lines, respectively, and there is an edge between two vertices if
and only if the corresponding point is incident with the corresponding line.  Throughout this paper, if $G$ is an incidence graph, then its vertices will be called points and lines, 
and the type of a vertex is either point or line, according to whether they correspond to a point or a line of the geometric structure. 

The incidence graph of a finite projective plane of order $q$ is a $(q+1)$-regular graph with girth 6
on $2(q^2+q+1)$ vertices.
Since the Moore bound for $k=q+1$ and $g=6$ is equal to the orders of these graphs, the incidence graphs of projective planes are $(q+1,6)$-cages.

\section{Lower and upper bounds on the order of balanced biregular graphs}\label{bounds}

First, we present some preliminary observations on babi-graphs.

\begin{definition}
Let $G$ be an $(r,s;g)$-babi-graph. A vertex of $G$ is called \emph{thin vertex} or
\emph{fat vertex} if its degree is $r$ or $s$, respectively.

The edges of $G$ are divided into three classes. A \emph{thin edge} connects two
thin vertices, a \emph{fat edge} connects two fat vertices, and a \emph{mixed edge}
connects one thin and one fat vertex. 
\end{definition}

\begin{proposition}\label{fatedges_Gyuri}
Let $2\leq r<s$ be integers. 
Any babi-graph of order $v$ has at least $\frac{v}{4}(s-r)$ fat edges. 
\end{proposition}

\begin{proof} 
Let $G$ be an $(r,s;g)$-babi-graph of order $v.$ Let $f,t,m$ denote the number of fat, thin, and mixed edges of $G$, respectively. 
Then, counting the number of edges emanating from the fat and thin vertices, we have
$$2f+m=\frac{v}{2}s, \quad 2t+m=\frac{v}{2}r.$$
Hence 
\begin{equation}
\label{fatedgenumber}
f=\frac{v}{4}(s-r)+t.
\end{equation}
Since $t\geq 0,$ this proves the statement.
\end{proof} 

For any fat vertex $x$ let $deg_f(x)$ denote the number of fat edges through $x.$
Since any fat edge joins two fat vertices, we have 
$$\sum _{x \text{ fat vertex}}deg_f(x)=2f.$$

\begin{proposition}\label{fatedges-2_Gyuri}
Let $2\leq r<s$ be integers and $G$ be an $(r,s;g)$-babi-graph of order $v.$ 
Then $G$ has a fat edge $xy$ for which $deg_f(x)+deg_f(y)\geq 2(s-r).$
\end{proposition}

\begin{proof} 
By Proposition \ref{fatedges_Gyuri}, $G$ has $\frac{v}{2}$ fat vertices and  
$\frac{v}{4}(s-r)+t$ fat edges. 
Hence we can estimate the average value of the sum of $deg_f(x)+deg_f(y)$ for fat edges using the
inequality between the arithmetic and quadratic means.
\begin{equation}\label{average}
\begin{aligned}
 & \frac{1}{f}\sum _{xy \text{ fat edge}}(deg_f(x)+deg_f(y))=
\frac{1}{f}\sum _{x \text{ fat vertex}}(deg_f(x))^2 \geq \\ 
 &\frac{1}{f} \frac{(\sum _{x \text{ fat vertex}}deg_f(x))^2}{\frac{v}{2}}=\frac{8f}{v} 
=2(s-r)+\frac{8t}{v}\geq 2(s-r).
\end{aligned}
\end{equation}
\end{proof}

If $r=1$ and $G$ is a $(1,s;g)$-babi-graph, then $G$ has no thin edge, hence it has $\frac{v}{4}(s-1)$ fat edges. Delete the thin vertices of $G$ and the edges emanating from them. The resulting graph, $G'$ 
has girth $g$, its order is $\frac{v}{2},$ its maximum degree is $s$ and it has $\frac{v}{4}(s-1)$ edges.
There are two possibilities. Either $G'$ is an $(s-1,g)$-graph or $G'$ has some vertices of degree $s$. 
In the latter case, $G'$ is not a regular graph.

This procedure also works in reverse, because adding leaves to the vertices of a graph does not change the girth of the graph. In the first case, if $G'$ is an $(s-1,g)$-graph of order $\frac{v}{2}$ and we glue a leaf in each of its vertices, then the resulting graph is a $(1,s;g)$-babi-graph of order $v.$ 
In the second case, if $G'$ has girth $g$, its order is $\frac{v}{2},$ its maximum degree is $s$ and it has 
$\frac{v}{4}(s-1)$ edges, then for each vertex $u$ we can glue $s-deg(u)$ leaves. The resulting graph $G$
has $\frac{v}{2}$ vertices of degree $s$ (these are the vertices of $G'$) and 
$$\sum _{u\in V(G')}(s-deg(u))=\frac{v}{2}s-\sum _{u\in V(G')}deg(u)=
\frac{v}{2}s-2\frac{v}{4}(s-1)=\frac{v}{2}$$
vertices of degree 1. The girth of $G$ is the same as the girth of $G'$, hence $G$ is a $(1,s;g)$-babi-graph.
Therefore, the problem of finding $(1,s;g)$-babi-cages is more difficult than the classical problem
of finding $(s-1,g)$-cages.


From now on, we assume that $r>1$.
First, we show the existence of babi-graphs 
for any $2\leq r<s$ and $3\leq g$. 

Let $G_1=(V_1,E_1)$ and $G_2=(V_2,E_2)$ be two graphs such that $V_1\cap V_2=\emptyset$, and
$x_1y_1\in E_1$ and let $x_2y_2\in E_2$ be two edges. Then the \emph{connection of $G_1$ and $G_2$
through switching $x_1y_1$ and $x_2y_2$} is the graph 
$$\Gamma =
\left( V_1\cup V_2, (E_1\cup E_2\cup \{ x_1x_2,y_1y_2\} )\setminus \{ x_1y_1, x_2y_2\}  \right).$$

This operation obviously preserves the degrees of all vertices. We claim that if both $G_1$ and $G_2$ 
have girth $g$, then the girth of $\Gamma $ is at least $g$. Let $\mathcal{C}$ be a cyle in $\Gamma $. If all vertices of $\mathcal{C}$ are either in $G_1$ or in $G_2$, then its length is at least $g$. If 
$\mathcal{C}$ contains the edge $x_1x_2$, then it must contain the edge $y_1y_2$. Since the length of
the shortest path joining $x_i$ and $y_i$ in $G_i \setminus \{ x_iy_i\} $ is at least $g-1$, the length of 
$\mathcal{C}$ is at least $1+(g-1)+1+(g-1)=2g$ in this case.

\begin{proposition}
\label{sokszoroz}
Suppose that there exists a $(k,g)$-graph of order $v$. Then, for all positive integers $n\in \mathbb{N}$, 
there exist $(k,g)$-graphs of order $nv$.
\end{proposition}

\begin{proof}
For $k=2$ the statement is obvious. Suppose that $k>2.$ Let $G$ a $(k,g)$-graph of order $v$ and $C$ be 
a girth cycle in $G$. Since $k>2,$ there exists an edge $xy$ of $G$ that is not contained in $C$. Take two copies of $G$ and create their connection through switching the two copies of $e$. The resulting graph,
$\Gamma _2$ is $k$-regular of order $2v$ and its girth is at least $g.$ Since $\Gamma _2$ contains $C$,
its girth is $g$. In the same way, for $i>1$ the connection of $G$ and $\Gamma _i$ through switching $e$ and an arbitrary edge of $\Gamma _i$ results in $\Gamma _{i+1}$, that is a $k$-regular of order 
$(i+1)v$ and of girth $g$. This proves the statement by induction.
\end{proof}

\begin{theorem}
\label{2ossze}
Let $2\leq r<s$, and $ 3\leq g$ be integers.  
Then there exists an $(r,s;g)$-babi-graph for any triple of parameters $(r,s;g)$, and its order is bounded by: 
$$n_{bb}(r,s;g)\leq 32\left( \sum_{t=1}^{g-2}(r-1)^t\right) \left( \sum_{t=1}^{g-2}(s-1)^t\right) .$$
\end{theorem}

\begin{proof}
It was proven by Sachs \cite{101}, that there exist $(k,g)$-cages for any pair of parameters $(k,g)$.
Let $R$ be an $(r,g)$-cage of order $v_r$ and $S$ be an $(s,g)$-cage of order $v_s$. 
By Lemma \ref{sokszoroz}, we can create an $(r,g)$-graph $G_r$ and an $(s,g)$-graph $G_s$, both of order
$v_rv_s$. Let $C$ be a girth cycle in $G_s$ and $e$ be an edge of $G_s$ that is not contained in $C$
(since $s>2$, $e$ exists). Then the connection of $G_r$ and $G_s$ through an arbitrary edge of $G_r$ and $e$ is an $(r,s;g)$-babi-graph of order $2v_rv_s$.

A general upper bound on the size of a $(k,g)$-cage for all $2\leq k$ and 
$3\leq g$ was proven by Erd\H os and Sachs \cite{44}:
$$n(k,g)\leq 4\sum_{t=1}^{g-2}(k-1)^t.$$
Thus, our construction presents the upper bound 
$$n_{bb}(r,s;g)\leq 2v_rv_s\leq 
32\left( \sum_{t=1}^{g-2}(r-1)^t\right) \left( \sum_{t=1}^{g-2}(s-1)^t\right) .$$
\end{proof}

In the next theorem, we give a lower bound on the order of $(r,s;g)$-babi-graphs, which is based on the natural generalization of the Moore tree. 

\begin{theorem}\label{lowerboundbabigraph}
Let $2\leq r<s$ be integers.
Then the order of an $(r,s;g)$-babi-graph is at least:
\begin{itemize}
\item
$1+s, \text{ if } g =3,$
\item
$2s, \text{ if } g =4,$
\item
$1+s + 
(rs+(s-r)^2-s)(1+(r-1)+\ldots + (r-1)^{\frac{g-5}{2}}),
\text{ if } g>3 \text{ odd}$,  
\item
$2\left( s+
(rs+(s-r)^2-2s+1)(1+(r-1)+\ldots + (r-1)^{\frac{g-6}{2}})\right)  ,
\text{ if } g>4 \text{ even.}$
\end{itemize}

\end{theorem}

\begin{proof}
Let $G$ be an $(r,s;g)$-babi-graph. 
For any two vertices $x$ and $y$ let $d(x,y)$ denote their minimal-path-length distance in $G$. 
For a vertex $x$ and a non-negative integer $i$ we let $G_i(x) = \{y \in V(G) \colon d(x,y)=i\}$.  For an edge $xy$ of $G ,$ let $D_i^j(x,y)=G_i(x)\cap G_j(y).$ 

First, suppose that $g$ is even.
Take a fat edge $xy$ such that $deg_f(x)+deg_f(y)\geq 2(s-r)$ (by Proposition \ref{fatedges-2_Gyuri},
$xy$ exists) and create the Moore tree ${\cal{T}}_{xy}$. For $g=4$ the tree ${\cal{T}}_{xy}$ has 
$2+2(s-1)=2s$ vertices, this proves the statement. If $g>4$, then  ${\cal{T}}_{xy}$ contains $x,y$ the
vertices in $D_1^2(x,y)\cup D_1^2(y,x)$ and the vertices in $D_2^3(x,y)\cup D_2^3(y,x)$. Since
\begin{equation}
\label{tree-4}
|D_1^2(x,y)\cup D_1^2(y,x)| = 2(s-1)
\end{equation}
 and 
$D_1^2(x,y)\cup D_1^2(y,x)$ contains $deg_f(x)+deg_f(y)-2\geq 2(s-r-1)$ fat vertices, we get
$$|D_2^3(x,y)\cup D_2^3(y,x)|\geq   2(s-r-1)(s-1)+2r(r-1)=2(rs+(s-r)^2 -2s +1).$$
For $i=3,4,\dots ,\frac{g-2}{2}$ the inequality
\begin{equation}
\label{tree-estimate-even}
|D_i^{i+1}(x,y)\cup D_i^{i+1}(y,x)|\geq   (r-1)|D_{i-1}^{i}(x,y)\cup D_{i-1}^{i}(y,x)|
\end{equation}
obviously holds, hence
\begin{equation}
\label{tree-even}
|D_i^{i+1}(x,y)\cup D_i^{i+1}(y,x)|\geq   2(rs+(s-r)^2 -2s +1)(r-1)^{i-2}.
\end{equation}
The vertex-set of the tree ${\cal{T}}_{xy}$ consists of $x,y$ and  $D_i^{i+1}(x,y)\cup D_i^{i+1}(y,x)$ for
$i=1,2,\dots , \frac{g-2}{2}$. Adding Equation (\ref{tree-4}) and Inequalities (\ref{tree-even}), we get the statement of the theorem for $g$ even.

Similar counting works when $g$ is odd. For $g=3$ the statement is immediate. If $g>3$, then again
take a fat edge $xy$ with $deg_f(x)+deg_f(y)\geq 2(s-r)$. 
We may assume that $deg_f(x)\geq s-r$. Create the Moore tree ${\cal{T}}_{x}$. 
Since $x$ is a fat vertex, 
\begin{equation}
\label{tree-3}
|G_1(x)| =s,
\end{equation}
it has at least $s-r$ fat neighbours, hence
\begin{equation}
\label{tree-5}
|G_2(x)|\geq (s-r)(s-1)+r(r-1)=rs+(s-r)^2-s
\end{equation}
and
\begin{equation}
\label{tree-odd}
|G_i(x)|\geq (r-1)|G_{i-1}(x)|\geq (rs+(s-r)^2-s)(r-1)^{i-2}
\end{equation}
holds for $i=3,4,\dots , \frac{g-1}{2}$. 
The vertex-set of the tree ${\cal{T}}_{x}$ consists of $x$ and $G_i(x)$ for
$i=1,2,\dots , \frac{g-1}{2}$. Adding Equation (\ref{tree-3}), Inequality  (\ref{tree-5}) and Inequalities (\ref{tree-odd}), we get the statement of the theorem for $g$ odd.
\end{proof}

Let us remark that if equality occurs in Theorem \ref{lowerboundbabigraph} for $g>4$, then 
the graph has no thin edge. If there is no thin edge, then Inequalities (\ref{tree-estimate-even})
and (\ref{tree-odd}) cannot be tight for $i\geq 3,$ because all neighbours of thin vertices must
be fat. So equality cannot occur if $g>6$.
In the next theorem, we examine the case of equality when $g=5$ and $6.$

\begin{theorem}\label{equal-5-6}
Let $2\leq r<s$ be integers.  
\begin{itemize}
\item
If $n_{bb}(r,s;5)=rs+(s-r)^2+1$, then $r=2$ and $s=3$.
\item
If $n_{bb}(r,s;6)=2(rs+(s-r)^2-s+1)$, then $1.935r < s < 2r-2$.
\end{itemize}
\end{theorem}

\begin{proof}
Let $G$ be an $(r,s;g)$-babi-graph. 
It follows from the proof of Theorem \ref{lowerboundbabigraph} that in the case of equality
$G$ has no thin edge, and $deg_f(x)+deg_f(y)=2(s-r)$ for any fat edge $xy$. 

First, consider the case $g=5$. If the order of $G$ is $rs+(s-r)^2+1$, then equality holds in 
(\ref{tree-5}), so $deg_f(w)=s-r$ for all fat vertex $w$. If $x$ is a fat vertex, then $G_1(x)$
contains $s-r$ fat and $r$ thin vertices. There is no thin edge, hence in $G_2(x)$ there are 
$$(s-r)(s-r-1)+r(r-1)=(s-r)^2+r^2-s$$
fat vertices. So the total number of fat vertices in $G$ is 
$$1+(s-r)+(s-r)^2+r^2-s=(s-r)^2+r^2-r+1.$$
On the other hand, half of the vertices are fat, thus
$$(s-r)^2+r^2-r+1=\frac{rs+(s-r)^2+1}{2}.$$
The solution of this quadratic equation is
$$s=\frac{3r\pm \sqrt{-3r^2+8r-4}}{2}.$$
The discriminant is negative if $r>2.$ If $r=2$, then $s=3$ and $v=8$.

The following construction shows the existence of a $(2,3;5)$-babi-cage of order $8$.
Let $\Gamma$ be the Petersen graph and take an edge $v_1v_2$ of it. Delete the vertices $v_i$ and the edges emanating from $v_i$ for $i=1,2.$ The neighbours of
the $v_i$s are pairwise different vertices, because the girth of $\Gamma $ is 5. 
The obtained graph $\Gamma '$ has 8 vertices, 4 of them (the
remaining neighbours of the $v_i$s) have degree 2, and 4 of them have degree 3.
So $\Gamma $ is a $(2,3;5)$-babi-cage.

\smallskip
Now, consider the $g=6$ case. Let $xy$ be a fat edge of $G$. We may assume without loss of generality
that $1\leq deg_f(x)=s-r-d$ and $s\geq deg_f(y)=s-r+d$ where $d$ is a non-negative integer, $d\leq r$ and $d\leq s-r-1$. Then for any fat neighbour $w$ of $x$ we have 
$$deg_f(w)=2(s-r)-deg_f(x)=s-r+d.$$
There are $s-r-d-1$ fat and $r+d$ thin vertices in $D_1^2(x,y)$. Since $G$ has no thin edge,
$D_2^3(x,y)$ contains
 $$(s-r-1-d)(s-r-1+d)+(r+d)(r-1)=(s-r-1)^2-d^2+(r+d)(r-1)$$
fat vertices. In the same way, we get that there are $s-r+d-1$ fat vertices in $D_1^2(y,x)$ and 
$$(s-r-1)^2-d^2+(r+d)(r-1)$$ 
fat vertices in $D_2^3(y,x)$. So the total number of fat vertices in $G$ is

$$2\left( (s-r)+(s-r-1)^2+r(r-1)-d^2 \right).$$
On the other hand, half of the vertices are fat, thus
$$2\left( (s-r)+(s-r-1)^2+r(r-1)-d^2 \right)=rs+(s-r)^2-s+1.$$
Rearranging this equation, we get
\begin{equation}
\label{diophantine}
s^2-(3r+1)s+4r^2-2d^2+1=0.
\end{equation}

The solution of this quadratic equation is
$$s=\frac{3r+1\pm \sqrt{8d^2-7r^2+6r-3}}{2}.$$
The discriminant is non-negative if 
$$d^2\geq \frac{7}{8}\left( \left( r-\frac{3}{7}\right) ^2+\frac{12}{49}\right) .$$
Thus 
$$d>\sqrt{\frac{7}{8}}\left( r-\frac{3}{7}\right)> 0.935r-0.401 .$$
Since $s-r-1\geq d,$ we get $s>1.935r.$
On the other hand, $d\leq r$ implies 
$$8d^2-7r^2+6r-3\leq r^2+6r-3 < (r+3)^2,$$
hence 
$$s < \frac{3r+1+(r+3)}{2}=2r+2.$$

If $s=2r+h$, then from Equation (\ref{diophantine}) we get
$d^2=r^2+\frac{h-2}{2}r+\frac{h^2-h+1}{2}$. Since the inequalities 
$$r^2-2r+1< r^2+\frac{h-2}{2}r+\frac{h^2-h+1}{2}< r^2$$
hold for $h=-2,-1,0$ and $1$, in these cases $r-1<d<r$, so Equation (\ref{diophantine}) has no integer solution. This proves the second inequality of the statement.
\end{proof}

\section{Balanced biregular cages of girth $3$ and $4$}\label{girth3and4}

In this section, we give the exact orders of babi-cages when $g\in \{3,4\}$. 
Recall that, the solution for the classical problem in these cases is easy, for any $k\geq 2$, the $(k,3)$-cage is the complete graph $K_{k+1}$, and the $(k,4)$-cage is the complete bipartite graph $K_{k,k}$.


\begin{theorem}
Let $2\leq r<s$ be integers. Then 
$$n_{bb}(r,s;3)=2(s-r+1), \text{ if }\frac{s+1}{2} \geq r.$$
If $r>\frac{s+1}{2} $, then
$$n_{bb}(r,s;3)=
\begin{cases}
s+1, \text{ if } s\equiv 3\pmod 4, \text { or } s\equiv 1\pmod 4 \text{ and } r \text{ is odd}, \\
s+2, \text{ if }  s\equiv 2\pmod 4, \text { or } s\equiv 0\pmod 4   \text{ and } r \text{ is even}, \\
s+3, \text{ if } s\equiv 1\pmod 4 \text{ and } r \text{ is even}, \\
s+4, \text{ if } s\equiv 0\pmod 4  \text{ and } r \text{ is odd}.
\end{cases}
$$
\end{theorem}

\begin{proof}
Let $G$ be an $(r,s;3)$-babi-graph of order $v.$ 
The number of fat edges is at most ${v/2 \choose 2},$ so by Proposition
\ref{fatedges_Gyuri} we have 
$${v/2 \choose 2}=\frac{1}{2}\frac{v}{2}\left( \frac{v}{2}-1\right) \geq \frac{v(s-r)}{4}.$$
This implies $v\geq 2(s-r+1).$

In the case $s+1 \geq 2r$ we construct an $(r,s;3)$-babi-cage of order $2(s-r+1)\geq s+1$. 
Let the vertices of $\Gamma $ be 
$$\{ u_0,u_1,\dots ,u_{s-r}\}  \cup  \{ w_0,w_1,\dots ,w_{s-r}\}  .$$
We define the edges of $\Gamma $ in the following way:  \\
i) Create all edges of type $u_iu_j$ for $0\leq i,j\leq s-r$. These edges and the corresponding vertices form a complete graph $K_{s-r+1}$ on the vertices $u_i$. \\
ii) Add edges of type $w_iu_{i+j}$ for all $0\leq i\leq s-r $ and all $0\leq j \leq r-1$ (the subscripts are calculated modulo $s-r+1$).   \\
By definition, the order of $\Gamma $  is $2(s-r+1)$.
Since $s-r+1\geq r$, the degrees of the vertices $u_i$ are $(s-r)+r=s$ and 
the degrees of the vertices $w_i$ are $r.$ So $\Gamma $ is an 
$(r,s;3)$-babi-cage of order $2(s-r+1)$.

Now, suppose that $r>\frac{s+1}{2} $.
The simple graph $G$ has some vertices of degree $s$, thus it has at least $s+1$
vertices. The order of $G$ is always even, and is divisible by $4$ if $s-r$ is odd. 
Hence $n_{bb}(r,s;3)\geq s+1, \, s+2, \, s+3$ or $s+4$ according to the cases listed in the statement of the theorem.

We construct an $(r,s;3)$-babi-cage of order $v=s+i$ in the $i$-th case listed in the statement. Notice that $v$ is even. Take 
the complete graph $K_{s+i}$ and divide its vertices into to subsets:
$$U=\{ u_1,u_2,\dots ,u_{v/2}\}  \text{ and }  W=\{ w_1,w_2,\dots ,w_{v/2}\}  .$$
The edge set of $K_{s+i}$ can be divided into three subsets. The edges of a complete graph
whose vertex set is $U$, the edges of a complete graph
whose vertex set is $W$, and the edges of a complete bipartite graph whose bipartitions are $U$
and $W$. Let   
$E_U=\{ u_iu_j \colon 1 \leq i\neq j \leq v/2 \} , E_W=\{ w_iu_w \colon 1 \leq i\neq  j \leq v/2 \} $, \\
$E_{UW}=\{ u_iw_j \colon 1 \leq i\neq j \} .$

 It is well-known that the complete graph $K_n$ can be decomposed into $n-1$ 1-factors if $n$ is even and into $\frac{n-1}{2}$ Hamilton cycles if $n$ is odd, and the complete bipartite graph $K_{n,n}$ can be decomposed into $n$ perfect matchings.

Decompose $E_{UW}$ into $v/2$ perfect matchings and in the $i$-th case delete $i-1$ of these matchings. After this deletion, the degree of each vertex becomes $s$. In the next step, we decrease
the degrees of the vertices in $W$ by $s-r$. We can easily do it if $v/2$ is even: decompose $E_W$ into
1-factors and delete $s-r$ of them. If $v/2$ is odd (the cases $s\equiv 1\pmod 4$ and $r$ is odd, and 
 $s\equiv 2\pmod 4$ and $r$ is even), then $s-r$ is even. We can decompose $E_W$ into
Hamilton cycles and delete $(s-r)/2$ of them. After the second step, the degree of each vertex 
in $W$ becomes $s-(s-r)=r$, the degrees of vertices in $U$ do not change, so we get 
an $(r,s;3)$-babi-cage of order $v=s+i$.
\end {proof}

\begin{theorem}
Let $2\leq r<s$ be integers. Then 
$$n_{bb}(r,s;4)=
\begin{cases}
4(s-r), \text{ if } s \geq 2r,  \\ 
2s, \text{ if } 2r>s \text{ and } s \text{ is even}, \\ 
2(s+1), \text{ if } 2r>s \text{ and } s \text{ is odd}.
\end{cases}
$$
\end{theorem}

\begin{proof}
Let $G$ be an $(r,s;4)$-babi-graph of order $v.$ 
By Theorem \ref{lowerboundbabigraph}, $v\geq 2s.$ 

The fat vertices and fat edges 
form a triangle-free subgraph of $G$. Thus, by Mantell's Theorem,
the number of fat edges is at most $\frac{(v/2)^2}{4},$ so by Proposition
\ref{fatedges_Gyuri} we have 
$$\frac{(v/2)^2}{4}\geq \frac{v(s-r)}{4}.$$
This implies $v\geq 4(s-r).$

First, suppose that $s \geq 2r$. Then $4(s-r)\geq 2s$. We construct an 
$(r,s;4)$-babi-cage of order $4(s-r)$. Let the vertices of $\Gamma _1$ be $U\cup L$, where 
$$
U=\{ u_0,u_1,\dots ,u_{(s-r)-1}\} \cup \{ x_1,x_2,\dots ,x_{s-r}\}, $$ 
$$L= \{ w_0,w_1,\dots ,w_{(s-r)-1}\}  \cup \{ y_1,y_2,\dots ,y_{s-r}\} .
$$
We define the edges of $\Gamma _1$ in the following way:  \\
i) Create all edges of type $x_iy_j$ for $1\leq i,j\leq s-r$. These edges and the corresponding vertices form a complete bipartite graph $K_{s-r,s-r}$. \\
ii) Add edges of type $x_iw_{s-r+i+j}$ for all $1\leq i\leq s-r $ and all $0\leq j \leq r-1$.  \\
iii) Add edges of type $y_iu_{s-r+i+j}$ for all $1\leq i\leq s-r$ and all $0\leq j \leq r-1$. \\ 
By definition, $\Gamma _1$ is a bipartite graph with bipartitions $U$ and $L$, its order is $4(s-r)$, and its girth is $4$. 
The degrees of the vertices $x_i$ and $y_i$ are $(s-r)+r=s$ and 
the degrees of the vertices $u_i$ and $w_i$ are $r.$ So $\Gamma _1$ is an 
$(r,s;4)$-babi-cage of order $4(s-r)$. 

In the second case $\frac{s}{2}>s-r$. Let $v=2s$. Since $s$ is even, $v$ is divisible by $4.$  Let the vertices of $\Gamma _2$ be $U\cup L$, where 
$$
U=\{ u_0,u_1,\dots ,u_{v/4-1}\} \cup \{ x_1,x_2,\dots ,x_{v/4}\}, $$ 
$$L= \{ w_0,w_1,\dots ,w_{v/4-1}\}  \cup \{ y_1,y_2,\dots ,y_{v/4}\} .
$$
Take the complete bipartite graph with bipartitions $U$ and $V$. Then delete 
the edges of type $u_iw_{i+j}$ for all $0\leq i\leq v/4-1 $ and all $0\leq j \leq s-r-1$. 
After the deletion, the degrees of the vertices $x_i$ and $y_i$ are $s$, and 
the degrees of the vertices $u_i$ and $w_i$ are $s-(s-r)=r.$ So $\Gamma _2$ is an 
$(r,s;4)$-babi-cage of order $2s$. 

In the third case, first we show that $n_{bb}\geq 2(s+1)$. Suppose to the contrary that $G$ is an $(r,s;4)$-babi-graph of order $2s$. Let $uv$ be a fat edge of $G$. Then the $s$ neighbours of $u$ and the $s$ neighbours of $v$ form a bipartition of the vertices of $G$, because $G$ does not contain any 3-cycle. Thus $G$ is a bipartite graph. Since $s$ is odd,
one of the two bipartitions contains more fat vertices than the other one. But this implies that the number of edges emanating from this bipartition is greater than the number of edges emanating from the other bipartition, a contradiction. Since the order of $G$ is an even number, it must be at least $2(s+1)$. Now we give a construction.  
If $s\leq 2r$ and $s$ is odd, then $s+1\leq 2r$ and $v=2(s+1)$ is divisible by $4$. Replace $s$ with $s+1$ and create the graph $\Gamma '_2$ in the same way as $\Gamma _2$ was created in the second case.
Then $\Gamma '_2$ is an $(r,s+1;4)$-babi-graph of order $2(s+1)$. Finally, delete the edges $x_iy_i$ of $\Gamma '_2$ for $1\leq i\leq (s+1)/2$. The graph $\Gamma _3$ obtained after deleting these edges is an $(r,s;4)$-babi-graph of order $2(s+1)$. 
\end{proof}

\section{Balanced biregular graphs of girth $5$}\label{girth5}

First, we consider semi-regular babi-graphs. Recall that a {\emph{semi-regular graph}} is a biregular graph with consecutive degrees. In the proof of Theorem \ref{equal-5-6} we have already constructed a 
$(2,3;5)$-babi-cage of order $8$.
Now, we consider the case $r>2$.

\begin{lemma}\label{semireg-5}
Let $r>2$ be an integer. Then 
$$n_{bb}(r,r+1;5)\geq
\begin{cases}
r^2+r+4, \quad\text{if $r\equiv 0,3 \pmod 4$,} \\
r^2+r+6, \quad\text{if $r\equiv 1,2 \pmod 4$.}
\end{cases}$$
\end{lemma}

\begin{proof}
By Theorems \ref{lowerboundbabigraph} and and \ref{equal-5-6}, we have 
$$n_{bb}(r,r+1;5) > r^2+r+2$$
for $r>2$. In a simple graph, the number of vertices of odd degree is an even number; hence, the order of an 
$(r,r+1;5)$-babi-graph must be divisible by $4.$ The smallest numbers greater than $r^2+r+2$ and divisible by $4$ are $r^2+r+4$ and $r^2+r+6$, respectively. This proves the statement.
\end{proof}

\subsection{Constructions of semi-regular babi-cages with girth $5$}

In this subsection, we present some babi-cages of girth $5.$ The construction of the $(2,3;5)$-babi-cage on 8 vertices in the proof of Theorem \ref{equal-5-6} was a simple deletion of some vertices and edges from the $(3,5)$-cage. This method also works for some other values of $r$. 

\begin{theorem}
There exists a
\begin{enumerate}
\item
$(3,4;5)$-babi-cage on 16 vertices,
\item 
$(5,6;5)$-babi-cage on 36 vertices,
\item 
$(6,7;5)$-babi-cage on 48 vertices.
\end{enumerate}
\end{theorem}

\begin{proof}
1. By Lemma \ref{semireg-5}, we have $n_{bb}(3,4;5)\geq 16.$
Let $G$ be the Robertson graph, the $(4,5)$-cage on 19 vertices. Take a 3-path $u_1u_2u_3$. Delete the vertices $u_i$ and the edges emanating from $u_i$ for $i=1,2,3.$ The neighbours of
the $u_i$s are pairwise different vertices, because the girth of $G$ is 5. 
The obtained graph $\Gamma $ has 16 vertices, 8 of them (the
remaining neighbours of the $u_i$s) have degree 3, and 8 of them have degree 4.
So $\Gamma $ is a $(3,4;5)$-babi graph. Its order is minimal, so it is a babi-cage.

2. By Lemma \ref{semireg-5}, we have $n_{bb}(5,6;5)\geq 36.$
Let $G$ be the $(6,5)$-cage on 40 vertices. The diameter of $G$ is 3. Let $u_1$ and $u_4$ be two vertices whose distance is 3. Take a 4-path $u_1u_2u_3u_4$. The neighbours of
the $u_i$s are pairwise different vertices, because the girth of $G$ is 5 and $u_1$ and $u_4$ have no
common neighbours. Delete the vertices $u_i$ and the edges emanating from $u_i$ for $i=1,2,3,4.$ The obtained graph $\Gamma $ has 36 vertices, 18 of them (the
remaining neighbours of the $u_i$s) have degree 5, and 18 of them have degree 6.
So $\Gamma $ is a $(5,6;5)$-babi graph. Its order is minimal, so it is a babi-cage.

3.  By Lemma \ref{semireg-5}, we have $n_{bb}(6,7;5)\geq 48.$
Let $G$ be the Hoffman-Singleton graph, the $(7,5)$-cage on 50 vertices. Take a 1-factor
$\mathcal{F}$ of $G.$ Let $u_1$ and $u_2$ be the two vertices of an edge of $\mathcal{F}$. Delete them and the edges emanating from $u_1$ and $u_2$. 
Then, after this deletion, out of the remaining 24 edges of $\mathcal{F}$, 12 will have one vertex of degree 5 and one of degree 6, and 12 will have both vertices of degree 6. In the second step, we delete 6 of the latter edges, so we decrease the degree of 12 vertices by 1. The obtained graph $\Gamma $ has 48 vertices, 24 of them (the vertices without deleted edges) have degree 7, and 24 of them have degree 6. So $\Gamma $ is a $(6,7;5)$-babi graph. Its order is minimal, so it is a babi-cage.
 \end{proof}

With simple deletion from a $(5,5)$-cage of order $30$, we can only create a $(4,5;5)$-babi-graph of order $28$, not of order $24.$ The following construction of a $(4,5;5)$-babi-cage of order $24$ is much more sophisticated.

\begin{theorem}
There exists a $(4,5;5)$-babi-cage on 24 vertices.
\end{theorem}

\begin{proof}
By Lemma \ref{semireg-5}, we have $n_{bb}(4,5;5)\geq 24.$
Let $R$ be the Robertson graph, the $(4,5)$-cage on 19 vertices. Denote its vertices by $1,2,\dots ,19.$
Then the adjacency table of $R$ is the following (see: House of Graphs \cite{HoG}, no. 1288-1):

\medskip
\begin{tabular}{|l|l|l|l|}
\hline
1: 2 3 4 5
&
2: 1 8 9 10
&
3: 1 12 14 16
&
4: 1 11 18 19
\\
5: 1 13 15 17
&
6: 7 14 17 19
&
7: 6 15 16 18
&
8: 2 13 14 18
\\
9: 2 11 16 17
&
10: 2 12 15 19
&
11: 4 9 14 15
&
12: 3 10 17 18
\\
13: 5 8 16 19
&
14: 3 6 8 11
&
15: 5 7 10 11
&
16: 3 7 9 13
\\
17: 5 6 9 12
&
18: 4 7 8 12
&
19: 4 6 10 13
&
\\
\hline
\end{tabular}

\medskip
The distance-3 graph of $R$ is a union of three cycles: $(11,12,13), \, (1,6,2,7)$ and
$(3,15,8,17,4,16,10,14,5,18,9,19).$ Let $C=(a,b,c,d,e)$ be another five-cycle. 

Now, we define a new graph, $G$.
Let the vertices of $G$ be the vertices of $R\cup C,$ the edges be the edges $38$ of $R$, the
$5$ edges of $C$ and $11$ new edges: 
$a11, a12, a13, b2, b6, c3, c15, d8, d17, e1, e7.$
Then the valency 5 vertices of $G$ are $1,2,3,6,7,8,11,12,13,15,17,a.$ The girth of $G$ is 5, because
The girth of $R$ is $4$, the new neighbours 
of a vertex of $C$ end adjacent vertices of the distance-3 graph of $R,$ and if two 
vertices are adjacent in $C$, then their new edges end in non-adjacent vertices of $R$. Hence $G$ does not contain any triangle or $4$-cycle. Since $G$ contains $C$, its girth is $5$, so $G$ is a
$(4,5;5)$-babi-cage on 24 vertices.  
\end{proof}

In the following theorem, we present a $(4,5;5)$-babi-graph on $28$ vertices.
Although this graph is not a cage, its order is close to that of a cage, and its simple geometric construction is worth showing.

\begin{theorem}
There exists a $(4,5;5)$-babi graph on 28 vertices.
\end{theorem}

\begin{proof}
Consider $\Pi _4=\mathrm{PG}(2,4)$. It can be partitioned into three disjoint Baer
subplanes (see \cite{bruck}), and each of these planes is isomorphic to the Fano plane,
$\mathrm{PG}(2,2)$. Let $\Pi_2$ denote one of these Fano planes. Let 
$\mathcal{P}=\{ P_i \colon i=1,2,\dots ,7\}$ be the set of points of $\Pi_2$. 
For any $P_i\in \mathcal{P}$, there are five lines of $\Pi_4$ through $P_i$; three
of them are lines of $\Pi_2$, the other two lines, say $e_i$ and $f_i$, are
tangents to $\Pi_2$. Thus, the seven points of $\Pi_2$ define a matching on the 14 
lines of $\Pi_4\setminus \Pi_2$.

Delete the points and lines of $\Pi_2$ from $\Pi_4$ and let $G$ be the Levi graph of the remaining point-line incidence geometry. Then $\Gamma $ is a bipartite $(4,6)$-graph on $2\times 14=28$ vertices. For $i=1,2,\dots,7$, add a new edge between the
vertices corresponding to the lines $e_i$ and $f_i$. Let $\Gamma $ denote the graph $G$
extended by the new edges. By definition, $\Gamma $ has 14 vertices of degree 4
(the points of $\Pi_4\setminus \Pi_2$) and 14 vertices of degree 5 (the lines of $\Pi_4\setminus \Pi_2$).

We claim that $\Gamma $ has girth $5.$ 
Suppose that $\Gamma $ contains a cycle $\mathcal{C}$ of length $3$ or $4$. Then 
$\mathcal{C}$ must contain a new edge. A new edge joins two lines of 
$\Pi_4\setminus \Pi_2$. These pairs of vertices have no common neighbours in 
$\Gamma,$ because their point of intersection is a point of $\Pi_2,$ so it was
deleted. Hence $\Gamma $ is triangle-free. If 
$\mathcal{C}$ is a four-cycle $(u_1,u_2,w_2,w_1)$ where $u_1u_2$ is a new edge, 
then $u_1$ and $u_2$ are lines. So $w_2$ and $w_1$ must be points. Since there is no edge between any two points, the girth of $\Gamma $ is at least $5.$ 
There are 5-cycles in $\Gamma $
(for example $(\ell_1,\ell_2,P_2,\ell _3,P_1)$, where $\ell _1\ell_2$ is a new 
edge, $\ell _3$ is a third tangent to $\Pi_2$, and $P_i$ is the point of intersection of $\ell_i$ and $\ell_3$
for $i=1,2$), so its girth is 5.
\end{proof}

\subsection{Constructions of small babi-graphs with girth $5$ using biaffine planes}

In this subsection, we employ a technique called \emph{amalgamation}, which was introduced by Funk in 2009  \cite{F2009} and Jorgensen \cite{jorg}. Recall that the amalgamation of a graph $\Gamma $ in a set $S$ consists of taking the elements of $S$ as the vertices of $\Gamma $ and add the edges of 
$\Gamma $ are drawn between the elements of $S$.

This method has been used in a lot of papers to construct regular and biregular graphs of girth $5$ (see, for instance, \cite{AABL11, AABB17, ABBM19, AB21}), as well as, more recently, mixed graphs of girth $5$ (see, for instance, \cite{ACG22, AM26, JJKP25}). 
Start with the incidence graph of a biaffine plane, choose specific sets of points or lines, and amalgamate graphs of the order of the cardinality of the chosen set.  In those papers, the authors are interested in amalgamating graphs that preserve some conditions related to the weights of the edges or with the direction of the arcs. In our case, it is sufficient to amalgamate regular graphs of girth at least $5$.

\begin{theorem}
Suppose that there exists a finite projective plane $\Pi _r$ of order $r$ and a $k$-regular, not necessarily connected graph $\Gamma $ with girth at least 5 and of order $v=r-1$ or $v=r$. Then there exists an 
$(r,r+k;5)$-babi-graph of order $v=2r^2-2$ or $v=2r^2$, respectively. This graph has $\frac{v}{4}k$ fat edges and no thin edges.  
\end{theorem}

\begin{proof}
Construct the type 1 biaffine plane $\mathcal{B}_1$ and the type 2 biaffine plane $\mathcal{B}_2$
from $\Pi _r$. For $i=1,2$ let $G_i$ be the Levi graph of $\mathcal{B}_i$
Then $G_i$ is an $(r,6)$-graph of order $v=2r^2$ if $i=1$ and of order $v=2r^2-2$ if $i=2$. 

Consider the deleted lines of $\Pi _r$ that contain some points of $\mathcal{B}_i$. There are $r$ 
or $r-1$ points of $\mathcal{B}_i$ on each of these lines, depending on whether $i=1$ or $i=2.$
Amalgamate the corresponding $\Gamma $ in each of the sets of points of the deleted lines.
Let $G_i'$ denote the graph obtained from $G_i$ by adding the new edges. By definition, $G_i'$ has
$v/2$ vertices of degree $r+1$ (the points of $\mathcal{B}_i$) and $v/2$ vertices of degree $r$ (the lines of $\mathcal{B}_i$).

We claim that $G_i'$ has girth $5.$ 
Suppose that $G_i'$ contains a cycle $\mathcal{C}$ of length $3$ or $4$. Then 
$\mathcal{C}$ must contain a new edge. A new edge joins two points of 
$\mathcal{B}_i$. These pairs of vertices have no common neighbour in 
$G_i$, because they are on a deleted line. Hence $G_i'$ is triangle-free. If 
$\mathcal{C}$ is a four-cycle $(u_1,u_2,w_2,w_1)$ where $u_1u_2$ is a new edge, 
then $u_1$ and $u_2$ are points. At least one of $w_2$ and $w_1$ must be line, because
$\Gamma $ does not contain a four-cycle. Without loss of generality, we may assume that $w_1$ is a line.
Since there is no edge between any two lines, this implies that $w_2$ is a point. There is an edge between
$u_1$ and $u_2,$ hence they are on the same deleted line. So do $u_2$ and $w_2$. Since we deleted
exactly one line through any remaining points, $u_1,u_2,w_2$ are on the same deleted line.
Thus $u_1$ and $w_2$ have no common neighbour in $G_i$, a contradiction. 

There are 5-cycles in $G_i'$ (for example $(P_1,\ell_1,P_3,\ell_2,P_2)$, where $P_1P_2$ is a new 
edge, $\ell _1$ and $\ell _2$ are non-parallel, non-deleted lines through $P_1$ and $P_2$, respectively,
and $P_3$ is their point of intersection), so its girth is 5.

Notice that $G_i'$ has no thin edge, because there is no edge between any pair of lines. The 
number of fat edges is $\frac{v}{4}k$, because, by definition, there are $k$ fat edges through any of the  $\frac{v}{2}$ points.
\end{proof}

In the simplest case, when $\Gamma $ is the union of independent edges, we get the following.

\begin{corollary}
Suppose that there exists a finite projective plane $\Pi _r$ of order $r$. Then there exists an 
$(r,r+1;5)$-babi-graph of order $v=2r^2$ if $r$ is even, and of order $v=2r^2-2$ if $r$ is odd. 
\end{corollary}

Note that if $r=2$ or $3$, this corollary gives a new construction of a $(2,3;5)$-babi-cage on 8
vertices or a $(3,4;5)$-babi-cage on 16 vertices. 

When $\Gamma $ is an $r$-cycle, we get the following.

\begin{corollary}
Suppose that there exists a finite projective plane $\Pi _r$ of order $r>5$. Then there exists an
$(r,r+2;5)$-babi-graph of order $2r^2-2$ by amalgamating cycles to the Levi graph of the type 2 
biaffine plane arising from $\Pi _r$.
\end{corollary}

The order of these graphs are greater than the bounds arising from Theorems \ref{lowerboundbabigraph} and \ref{equal-5-6}:  
$$n_{bb}(r,r+2;5)\geq
\begin{cases}
r^2+2r+7, \quad\text{if $r$ is odd,} \\
r^2+2r+6, \quad\text{if $r$ is even.}
\end{cases}$$

For $r=2$ there exists a $(2,4;5)$-babi-cage of order $14$, see Figure \ref{(2,4;5)babicage}.
\begin{figure}[h]
\begin{center}
\includegraphics[scale=0.3]{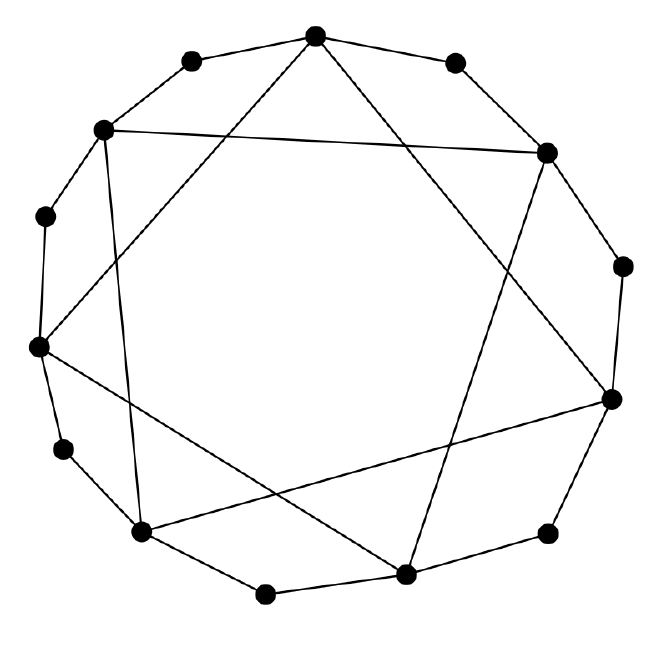}
\caption{A $(2,4;5)$-babi-cage of order $14$ (House of Graphs \cite{HoG}, no. 53705)}
\label{(2,4;5)babicage}
\end{center}
\end{figure}

For $r=3$ there exists a $(3,5;5)$-babi-graph of order $28.$ It can be constructed by 
simple deletion from one of the $(5,5)$-cages of order $30$, the Robertson-Wegner graph 
(for short, $RW$) \cite{111}. 

For the sake of completeness, we recall the geometric construction of $RW$.
Let $\mathcal{D}$ be a regular dodecahedron. The  vertices  of $\mathcal{D}$
determine five cubes. Each of these cubes determines two regular tetrahedra. The vertices
of $RW$ are the 20 vertices of $\mathcal{D}$, plus one vertex for each of the 10
tetrahedra.  The edges of $RW$ are the 30 edges of $\mathcal{D}$, plus each of the tetrahedral vertices is adjacent to its four determining vertices; finally, two tetrahedral vertices are adjacent if they are contained in the same cube.

Let $A_1B_1C_1D_1$ and $A_2B_2C_2D_2$ be two opposite faces of a cube inscribed into $\mathcal{D}$. Delete the two tetrahedral vertices
$A_1B_2C_1D_2$ and $A_2B_1C_2D_1$, and the edges emanating from them. The resulting graph, $RW'$, 
has 28 vertices, 20 of which have degree 5 and 8 vertices (the vertices of the cube) of degree 4. For $i=1,2$ the square $A_iB_iC_iD_i$ uniquely determines an edge $E_iF_i$ of $\mathcal{D}$ such that $A_iE_i,\, B_iE_i,\, C_iF_i,\, D_iF_i$ are edges of $\mathcal{D}$.
Furthermore, each of $A_2E_2B_2$ and $C_2F_2D_2$ determines a face of $\mathcal{D}$, $A_2E_2B_2G_2H_2$ and  $C_2F_2D_2K_2L_2$. In the second step, delete from $RW'$ the edges $A_1E_1,\, B_1E_1,\, C_1F_1,\, D_1F_1,\, E_1F_1$ and $A_2H_2,\, H_2G_2,\, G_2B_2,\, 
C_2L_2,\, L_2K_2,\, K_2D_2$. The second deletion reduces the degree of the vertices of the cube by one and the degrees of $E_1,\, F_1,\, G_2,\, H_2,$
$\, K_2, \, L_2$ by two. So the resulting graph is a $(3,5;5)$-babi-graph of order $28.$

\smallskip
We can construct relatively low-order babi-graphs using cages for the amalgamation graphs. We highlight two nice examples. We can amalgamate the Petersen graph into the type 2 biaffine plane arising from 
$\mathrm{PG}(2,11)$, and we can amalgamate the Robertson graph 
into the type 1 biaffine plane arising from $\mathrm{PG}(2,19)$.  We get the following.

 \begin{corollary}
$ $ \begin{itemize}
\item
There exists a $(11,14;5)$-babi-graph of order $240$.
\item
There exists a $(19,23;5)$-babi-graph of order $722$.
\end{itemize}
\end{corollary}

Note that the bounds arising from Theorems \ref{lowerboundbabigraph} and \ref{equal-5-6}
are $n_{bb}(11,14;5)\geq 152$ and $n_{bb}(19,23;5)\geq 456$.

\section{Balanced biregular graphs of girth $6$}\label{girth6}

As in the previous section, we begin by examining semi-regular graphs.
By Theorems \ref{lowerboundbabigraph} and and \ref{equal-5-6} , we have 
$$n_{bb}(r,r+1;6) > 2(r^2+1).$$
We give a slight improvement of this general bound for semi-regular graphs.

\begin{proposition}
\label{estimatewith-c}
Let $c > 2$ be an integer and $\Gamma $ be an $(r,r+1;6)$-babi-graph
of order $v=2r^2+c.$ Then $\Gamma $ has less than $\frac{cv}{8}$ fat edges and
less than $\frac{(c-2)v}{8}$ thin edges.
\end{proposition}

\begin{proof}
Let $xy$ be a fat edge of $\Gamma $. Then $deg_f(x)+deg_f(y) < c,$ because 
$deg_f(x)+deg_f(y) \geq c$
would imply that the Moore tree starting from the edge $xy$ should contain 
at least
$$2+2r+((c-1)r+(2r-(c-1))(r-1)=2r^2+c+1>v$$
vertices. Hence, by Inequality (\ref{average}), we get $c >\frac{8f}{v}.$ Thus
$f < \frac{cv}{8}$, and Equation (\ref{fatedgenumber}) implies $t < \frac{(c-2)v}{8}$. 
\end{proof}

\begin{theorem}\label{lowergirth6}
Let $r\geq 2$ be an integer. Then 
$$n_{bb}(r,r+1;6)\geq
\begin{cases}
2(r^2+2), \quad\text{if $r=2,4$ or $6$,} \\
2(r^2+3), \quad\text{if $r$ is odd,} \\
2(r^2+4), \quad\text{if $r\geq 8$ is even.}
\end{cases}$$
\end{theorem}

\begin{proof}
Suppose that $\Gamma $ is an $(r,r+1;6)$-babi-graph
of order $v=2r^2+c$. Then $c>2$. Since the order of a semi-regular babi-graph is divisible by $4$,
and the smallest numbers greater than $2r^2+2$ and divisible by $4$ are $2r^2+4$ if $r$ is even
and $2r^2+6$ if $r$ is odd, we proved the first and second statements of the theorem. 
 
Now, suppose that $r$ is even and the order of $\Gamma $ is $2r^2+4$.
Consider the Moore tree starting from a fat edge $xy.$ Among the neighbours of $x$ and $y$ there are at least $2r-2$ thin vertices. On the next level, there are 
$$(2r-2)(r-1)=2r^2-4r+2$$
neighbours of these thin vertices. Among them, the maximum number of fat vertices is
$v/2-2\leq r^2.$ Thus $\Gamma $ has at least 
$$(2r^2-4r+2)-r^2=r^2-4r+2$$ 
thin edges. On the other hand, by Proposition \ref{estimatewith-c}, 
$\Gamma $ has less than $\frac{v}{4}=\frac{r^2}{2}+1$ thin edges. 
Since the inequality
$$r^2-4r+2>\frac{r^2}{2}+1$$
holds for $r\geq 8,$ this contradiction proves the statement when $r$ is even.
\end{proof}


In the next theorems, we will give general constructions of balanced semiregular graphs and cages using the incidence graphs of projective planes. Recall that the order of every known finite projective plane is a prime power, and for every prime power $q$ there exists a plane of order $q$.

\begin{theorem}
Suppose that there exists a finite projective plane of order $r$. Then there exists an 
$(r,r+1;6)$-babi-graph of order $2(r^2+r).$
\end{theorem}

\begin{proof}
Let $\Pi _r$ be a finite projective plane of order $r$ and $\Gamma $ be the Levi graph of 
$\Pi _r$. Take a non-incident point-line pair $(P, \ell )$ in $\Pi _r$ and delete the corresponding two vertices of $\Gamma $. The remaining graph $\Gamma '$ has $2(r^2+r)$ vertices, $2(r+1)$ of degree $r$ (these are the points on $\ell $ and the lines
through $P$), and $2(r^2-1)$ vertices of degree $r+1$. Deleting some independent 
edges among the degree $r+1$ vertices, the resulting graph $G$ will be an  
$(r,r+1;6)$-babi-graph of order $2(r^2+r)$.
\end{proof}

Combining this construction and the bound given in Theorem \ref{lowergirth6} we get the following.

\begin{corollary}
$ $ \begin{itemize}
\item
Starting from the Heawood graph, the Levi graph of the Fano plane, we get
a $(2,3;6)$-babi-cage of order $12$.
\item
Starting from the Levi graph of $\mathrm{PG}(2,3)$, we get
a $(3,4;6)$-babi-cage of order $24$.
\end{itemize}
\end{corollary}

There exist non-isomorphic $(2,3;6)$-babi cages of order $12$. The graph given in the previous corollary 
can also be constructed in the following way:
take two disjoint $6$-cycles $C_1=(0,1,2,3,4,5)$ and $C_2=(0',1',2',3',4',5')$, and add the edges
$\{00',22',44'\}$. This graph has 3 fat edges and no thin edge. The $(2,3;6)$-babi cage shown in Figure 
\ref{(2,3;6)babicage}, however, has 6 fat edges and 3 thin edges.

\begin{figure}[h]
\begin{center}
\includegraphics[scale=0.3]{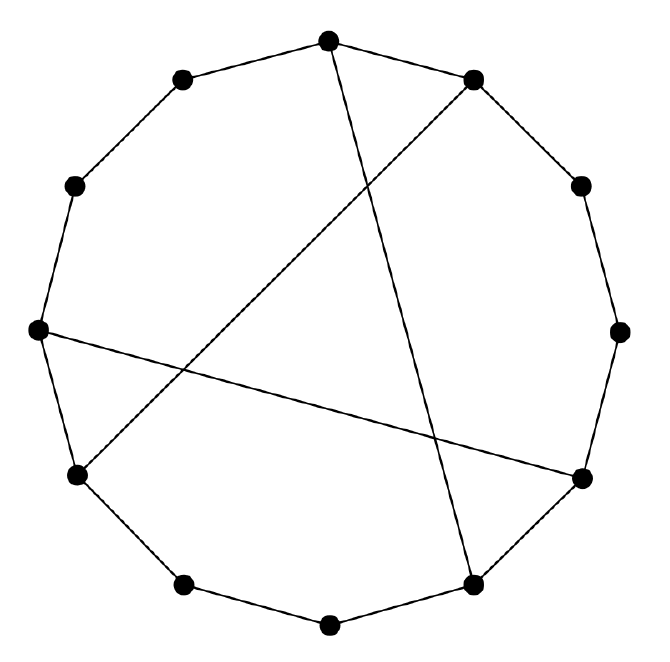}
\caption{A $(2,3;6)$-babi-cage of order $12$ (House of Graphs \cite{HoG}, no. 54321)}
\label{(2,3;6)babicage}
\end{center}
\end{figure}

\begin{theorem}
Suppose that there exists a finite projective plane of order $r >3$. Then there exists an 
$(r,r+1;6)$-babi-graph of order $2(r^2+r-2).$
\end{theorem}

\begin{proof}
Let $\Pi _r$ be a finite projective plane of order $r$ and $\Gamma $ be the Levi graph of 
$\Pi _r$. Take three non-collinear points, $P_1, \, P_2$ and $P_3$ in $\Pi _r$ and delete the vertices corresponding to these points and the three lines $P_1P_2, \, P_2P_3$ and $P_3P_1$ of $\Gamma $. The remaining graph $\Gamma '$ has $2(r^2+r-2)$ vertices, $6(r-1)$ of degree $r$ (these are the points on the lines $P_iP_j$ and the lines
through the points $P_i$), and $2(r^2-2r+1)$ vertices of degree $r+1$. Since $r>3,$ the number of degree $r+1$ vertices is greater than or equal to the number of degree
$r$ vertices. Deleting some independent edges among the degree $r+1$ vertices, the resulting graph $G$ will be an $(r,r+1;6)$-babi-graph of order $2(r^2+r-2)$.
\end{proof}

Again, combining this construction and the bound given in Theorem \ref{lowergirth6}, we get the following.

\begin{corollary}
$ $ \begin{itemize}
\item
Starting from the Levi graph of $\mathrm{PG}(2,4)$, we get
a $(4,5;6)$-babi-cage of order $36$.
\item
Starting from the Levi graph of $\mathrm{PG}(2,5)$, we get
a $(5,6;6)$-babi-cage of order $56$.
\end{itemize}
\end{corollary}

\begin{theorem}\label{rcong1mod4}
If $r\equiv 1 \pmod 4$ and there exists a finite projective plane of order $r$, then there exists an
$(r,r+1;6)$-babi-graph of order $2r^2+r+1$
\end{theorem}

\begin{proof}
Let $\Pi _r$ be a finite projective plane of order $r$ and $\Gamma $ be the Levi graph of 
$\Pi _r$. Take an incident point-line pair $(P,\ell )$. 
Delete from $\Gamma $ the vertices 
corresponding to $\ell$ and all the points on $\ell$ except $P$. Then the resulting graph,
$\Gamma'$ has $r^2+1$ vertices of degree $r$ (these are the $r^2$ lines not through $P$ 
and the point $P$) and $r^2+r$ vertices of degree $r+1$ (these are the points except $P$ and
the $r$ lines through $P$).  

Since $r\equiv 1 \pmod 4$, we can take $\frac{r-1}{4}$ lines through $P$ (except $l$) and one point different from $P$ on each of them. Delete from $\Gamma'$ the edges that join the selected 
incident point-line pairs. The new graph, $\Gamma''$, has $r^2+1+2\frac{r-1}{4}=\frac{2r^2+r+1}{2}$ vertices of degree $r$ and $r^2+r-2\frac{r-1}{4}=\frac{2r^2+r+1}{2}$ vertices of degree $r+1$. Its girth 
remained $6,$ hence $\Gamma''$ is a $(r,r+1;6)$-babi-graph of order $2r^2+r+1$. 
\end{proof}

Applying Theorems \ref{lowergirth6} and  \ref{rcong1mod4} for $r=5,$ we obtain the following:
\begin{corollary} There exits a $(5,6;6)$-babi-cage of order $56$. 
\end{corollary}

Our final construction is based on the geometric properties of ovals in planes of odd order.
Definitions and proofs of the properties of ovals used can be found, for example, in \cite[Lemmas 6.12, 6.14]{KSz}.

\begin{theorem}
Suppose that there exists a finite projective plane of odd order $r+2 >3$ 
that contains an oval (a set of $r+3$ points, no three of them are collinear).
Then there exists an 
$(r,r+2;6)$-babi-graph of order $2(r^2+3r+2).$
\end{theorem}

\begin{proof}
Let $\mathcal{O}$ be an
oval in the projective plane $\Pi _{r+2}.$ There are $r+3$ tangents to $\mathcal{O}$, one at each of its points, The number of secant lines (lines containing two points of $\mathcal{O}$) is $\frac{(r+3)(r+2)}{2}$, and there are $\frac{(r+1)(r+2)}{2}$ external lines (lines containing no points of $\mathcal{O}$). Any point of $\Pi _{r+2}\setminus \mathcal{O}$ lies on either $2$ or $0$ tangents to $\mathcal{O}.$
A point of $\Pi _{r+2}\setminus \mathcal{O}$
is called \emph{external} or \emph{internal} with respect to $\mathcal{O}$ according
to whether it lies on $2$ or $0$ tangents to $\mathcal{O}.$
There are $\frac{(r+3)(r+2)}{2}$ external
and $\frac{(r+1)(r+2)}{2}$ internal points. 
A tangent line contains one point of $\mathcal{O}$ and $r+2$ external points. If the line $\ell $ is not a tangent to $\mathcal{O}$, then exactly half of the points of  $\ell \setminus \mathcal{O}$ are external and half of them are internal.

Let $\Gamma $ be the Levi graph of 
$\Pi _{r+2}$. Take a point $(P \in  \mathcal{O}$ and let $t_p$ denote the tangent line to 
$\mathcal{O}$ at $P$. Delete from $\Gamma $ the points of $\mathcal{O}\cup t_P$, the tangents to  
$\mathcal{O}$ and all 
other lines through $P$. The order of the resulting graph, $\Gamma'$, is 
$$2((r+2)^2+(r+2)+1)-(2r+5)-(r+3)-(r+2)=2(r^2+3r+2).$$ 

If $R$ is a remaining point, then we deleted the line $PR$ and $0$ or $2$ tangents through $R$, depending on whether $R$ is an internal or an external point. Thus, the degree of $R$ is either $r+2$ or $r$. If $\ell $ is a remained line, then we deleted the point $\ell \cap t_P$ and the $0$ or $2$ points of
$\ell \cap \mathcal{O}$, depending on whether $\ell $ is an external line or a secant. Thus, the degree of $\ell $ is either $r+2$ or $r$. The number of remaining external points is equal to the number of the remaining secants, each of which is $\frac{(r+3)(r+2)}{2}-(r+2)=\frac{(r+1)(r+2)}{2}$. 
The number of remaining internal points and external lines is the same, since we have not deleted any of them. So $\Gamma'$ has $2\frac{(r+1)(r+2)}{2}=r^2+3r+2$ vertices of degree $r$ and the same number of vertices of degree $r+2$. The deletion does not decrease or increase the girth, hence  $\Gamma'$ is an $(r,r+2;6)$-babi-graph of order $2(r^2+3r+2).$
\end{proof}

A typical example for ovals in $\mathrm{PG}(2,q)$ is the projective extension of the normal parabola with affine equation $Y=X^2$. It has one point on the line at infinity, that is, the point at infinity of the vertical lines. If we choose this point as the point $P$ of the construction described in the previous theorem, then we delete all points of the parabola, the line at infinity, the tangents of the parabola, and the vertical lines. Straightforward calculations show that an affine point $(a,b)$ is external or internal,
according to whether $a^2+4b$ is a square or a non-square element  in  $\mathrm{GF}(q)^*$,
the multiplicative group of $\mathrm{GF}(q)$, and an affine line $Y=mX+k$ is secant or external, according to whether $m^2-4k$ is a square or a non-square element in $\mathrm{GF}(q)^*$.
 
\section{Conclusions and Future Work}\label{conclandfuturework}
As stated in the introduction, it is of interest to revisit the classical cage problem in the context of balanced biregular cages. In contrast to previous approaches to biregular or bipartite biregular cages, the graphs considered here exhibit a balance between vertices of different degrees. This feature leads to a fundamentally different problem, closely related to balanced networks with specific structural properties, or more generally, to the theory of balanced graphs.

\medskip
Regarding future work, several research directions can be proposed:

\begin{itemize}
\item {\bf{Problem 1:}} Continue the study of this problem by determining exact values for girth $5$ and $6$, using properties of previously known cages, as well as projective planes and biaffine planes, together with their incidence graphs and amalgamation techniques. 
Additionally, configurations arising in projective planes, such as ovals and hyperovals, or group-theoretic properties, may be exploited.

\item {\bf{Problem 2:}} Extend the study of this problem to vertex-transitive graphs and Cayley graphs.

\item {\bf{Problem 3:}} Determine bounds and constructions for balanced biregular graphs and cages with girth $g \in {7,8,11,12}$ using generalized $n$-gons for $n \in {4,6}$.

\item {\bf{Problem 4:}} Establish general lower and upper bounds using alternative techniques.

\end{itemize}

\noindent
{\bf Gabriela Araujo-Pardo:} 
Instituto de Matem\'aticas-Campus Juriquilla, Universidad Nacional Aut\'onoma de 
M\'exico, C.P. 076230, Boulevard Juriquilla \# 3001, Juriquilla, Qro., M\'exico;  \\
 e-mail:  \texttt{garaujo@im.unam.mx} \\

\noindent
{\bf Gy\"orgy Kiss:} Department of Geometry,  E\"otv\"os Lor\'and University, 1117 Budapest, P\'azm\'any
s. 1/c, Hungary; and Faculty of Mathematics, Natural Sciences and Information Technologies, University of Primorska, Glagolja\v ska 8, 6000 Koper,
Slovenia; \\
e-mail: \texttt{gyorgy.kiss@ttk.elte.hu} \\

\end{document}